\documentclass{amsart}
\usepackage{graphicx}
\usepackage{ulem}
\usepackage{bm}
\usepackage{color}
\usepackage{amssymb, amsmath, amsthm}
\usepackage{bm}
\theoremstyle{plain}
\usepackage{graphicx}
\usepackage[aboveskip=-5pt,position=bottom]{caption}
\usepackage{natbib}
\usepackage{url}
\usepackage{enumerate}
\usepackage{rotating}
\usepackage{booktabs,arydshln}
\usepackage{caption} 
\usepackage{float}

\usepackage{accents}

\newtheorem{lemma}{Lemma}[section]

\newtheorem{theorem}[lemma]{Theorem}

\newtheorem{prop}[lemma]{Proposition}
\newtheorem{exam}[lemma]{\normalfont \scshape
 Example}
\newtheorem{rem}[lemma]{\normalfont \scshape Remark}

\newcommand{\R}{\mathbb{R}}
\newcommand{\N}{\mathbb{N}}

\newcommand{\norm}[1]{\left\Vert#1\right\Vert}

\newcommand{\abs}[1]{\left\vert#1\right\vert}
\newcommand{\set}[1]{\left\{#1\right\}}

\newcommand{\bfx}{\bm{x}}
\newcommand{\bfzero}{\bm{0}}

\newcommand{\bfone}{\bm{1}}
\newcommand{\bfa}{\bm{a}}
\newcommand{\bfb}{\bm{b}}

\newcommand{\bft}{\bm{t}}
\newcommand{\bfU}{\bm{U}}
\newcommand{\bfu}{\bm{u}}

\newcommand{\bfy}{\bm{y}}

\newcommand{\calS}{\mathcal S}

\usepackage[doublespacing]{setspace}

\allowdisplaybreaks[4]

\usepackage{hyperref}
\hypersetup{colorlinks,%
citecolor=blue,%
filecolor=green,%
linkcolor=red,%
urlcolor=violet,%
}

\begin{document}

\title[Conditional Tail Independence]{Conditional Tail Independence in Archimedean Copula Models}

\author{Michael Falk\and Simone Padoan\and Florian Wisheckel}
\address[1,3]{Institute of Mathematics, University of W\"{u}rzburg, W\"{u}rzburg, Germany}
\email{michael.falk@uni-wuerzburg.de}\email{florian.wisheckel@uni-wuerzburg.de}
\address[2]{Department of Decision Sciences,
Bocconi University of Milan,  Milano, Italy}
\email{simone.padoan@unibocconi.it}



\begin{abstract}
Consider a random vector $\bfU$, whose distribution function coincides in its upper tail with that
of an Archimedean copula. We report the fact that the conditional distribution of $\bfU$,
conditional on one of its components, has under a mild condition on the generator function
independent upper tails, no matter what the unconditional tail behavior is. This finding is extended to Archimax copulas.

\end{abstract}

\maketitle


\section{Introduction}

Let $\bfU=(U_1,\dots,U_d)$ be a random vector (rv), whose distribution function (df) $F$ is in the
domain of attraction of a multivariate extreme value df $G$, denoted by $F\in\mathcal D(G)$, i.e.,
there are constants $\bfa_n=(a_{n1},\dots,a_{nd})>\bfzero\in\R^d$,
$\bfb_n=(b_{n1},\dots,b_{nd})\in\R^d$, $n\in\N$, such that for each $\bfx=(x_1,\dots,x_d)\in\R^d$
\[
F^n(\bfa_n\bfx+\bfb_n) \to_{n\to\infty} G(\bfx).
\]
Note that all operations on vectors such as $\bfx+\bfy$, $\bfx\bfy$ etc. are always meant
componentwise.

The rv $\bfU$, or, equivalently, the df $F$, is said to have asymptotically independent (upper)
tails, if
\[
G(\bfx)=\prod_{i=1}^d G_i(x_i),
\]
where $G_i$, $1\le i\le d$, denote the univariate margins of $G$.

We require in this paper that the df F of $\bfU$ coincides in its upper tail with a \emph{copula},
say $C$, i.e., there exists $\bfu_0=(u_{01},\dots,u_{0d})\in(0,1)^d$ such that
 \[
 F(\bfu)=C(\bfu),\qquad \bfu\in[\bfu_0,\bfone]\subset\R^d.
 \]
 Each univariate margin of a copula is the uniform distribution $H(u)=u$, $0\le u\le 1$, and, thus,
 each univariate margin of $F$ equals $H(u)$ for $u\in[v_0,1]$, where $v_0:=\max_{1\le i\le
 d}u_{0i}$.

 The significance of copulas is due to Sklar's theorem (\citet{sklar59, sklar96}), by which an
 arbitrary multivariate df  can be represented as a copula together with its univariate margins. The
 dependence structure among the margins of an arbitrary rv is, therefore, determined by the copula.
 For an introduction to copulas we refer to \citet{nel06}.

We require in this paper that the upper tail of $C$ is that of an \emph{Archimedean copula}
$C_\varphi$, i.e., there exists a convex and strictly decreasing function
$\varphi:(0,1]\to[0,\infty)$ with $\varphi(1)=0$, such that
\[
C_\varphi(\bfu) =\varphi^{-1}\left(\varphi(u_1)+\dots+\varphi(u_d)\right)
\]
for $\bfu\in[\bfu_0,\bfone]\subset\R^d$, where $\bfu_0=(u_{01},\dots,u_{0d})\in(0,1)^d$.

A prominent example is $\varphi_p(s):=(1-s)^p$, $s\in[0,1]$, where $p\ge 1$. In this case we obtain
\begin{equation}\label{eqn:definition of logistic copula}
C_{\varphi_p}(\bfu)= 1-\left(\sum_{i=1}^d(1-u_i)^p \right)^{1/p}, \qquad \bfu\in[\bfu_0,\bfone].
\end{equation}

Note that
\[
C_{\varphi_p}(\bfu):=\max\left(0,1-\left(\sum_{i=1}^d(1-u_i)^p \right)^{1/p}\right),\qquad
\bfu\in[0,1]^d,
\]
defines a multivariate df only in dimension $d=2$, see, e.g., \citet[Examples 2.1, 2.2]{mcnn09}. But
one can find for arbitrary dimension $d\ge 2$ a rv, whose df satisfies equation
\eqref{eqn:definition of logistic copula}, see, e.g., \citet[(2.15)]{falk2019}. This is the reason,
why we require the  Archimedean structure of  $C_\varphi$ only on some upper interval
$[\bfu_0,\bfone]$ and we do not speak of $C_\varphi$ as a \emph{copula}, but rather of a
\emph{distribution function}.

The behavior of $C_\varphi(\bfu)$ for $\bfu$ close to $\bfone\in\R^d$ determines the upper tail
behavior of the components of $\bfU$. Precisely, suppose that $C_\varphi\in\mathcal D(G)$, i.e.,
\[
C_\varphi\left(\bfone+\frac{\bfx}n\right)^n\to_{n\to\infty} G(\bfx),\qquad \bfx\le\bfzero\in\R^d,
\]
where the norming constants are prescribed by the univariate margins of $C_\varphi$, which is the
df $H(u)=u$, $u\in[v_0,1]$. We obviously have for arbitrary $x\le 0$ and $n$ large enough
\[
H\left(1+\frac xn \right)^n = \left(1+\frac xn \right)^n \to\exp(x).
\]
The multivariate max-stable df $G$, consequently, has standard negative exponential margins
$G_i(x)=\exp(x)$, $x\le 0$.

Moreover, there exists a norm $\norm\cdot_D$ on $\R^d$, such that $G(\bfx)=\exp(-\norm{\bfx}_D)$,
$\bfx\le \bfzero\in\R^d$; see, e.g., \citet{falk2019}. This norm $\norm\cdot_D$ describes the
asymptotic tail dependence of the margins of $C_\varphi$; the index $D$, therefore, means
\emph{dependence}. In particular $\norm\cdot_D=\norm\cdot_1$ is the case of (asymptotic)
independence of the margins, whereas $\norm\cdot_D=\norm\cdot_\infty$ yields their total dependence.
For the df $C_{\varphi_p}$ in \eqref{eqn:definition of logistic copula} we obtain, for example, for
$n$ large,
\begin{align*}
C_{\varphi_p}\left(\bfone+\frac{\bfx}n \right)^n &= \left(1-\frac
1n\left(\sum_{i=1}^d\abs{x_i}^p\right)^{1/p} \right)^n\\
&\to_{n\to\infty} \exp\left(-\norm{\bfx}_p \right),\qquad \bfx=(x_1,\dots,x_d)\le\bfzero\in\R^d,
\end{align*}
where $\norm{\bfx}_p=\left(\sum_{i=1}^d\abs{x_i}^p\right)^{1/p}$, $p\ge 1$, is the logistic norm on
$\R^d$. In this case we have tail independence only for $p=1$.

In this paper we investigate the problem, if conditioning on a margin $U_j=u$ has an influence on
the tail dependence of the left margins $U_1,\dots,U_{j-1},U_{j+1},\dots,U_d$. Actually, we will
show that the rv $(U_1,\dots,U_{j-1},U_{j+1},\dots,U_d)$, conditional on $U_j=u$, has in general
independent tails, for each choice of $j$, no matter what the unconditional tail behavior is; see
Section \ref{sec:main theorem}. This is achieved under a mild condition on the generator function
$\varphi$, which is introduced in Section \ref{sec:condition on the generator function}.

\section{Condition on the generator function}\label{sec:condition on the generator function}

Our results are achieved under the following condition on the generator function $\varphi$. There
exists a number $p\ge 1$ such that
\[
\lim_{s\downarrow 0}\frac{\varphi(1-sx)}{\varphi(1-s)}=x^p,\qquad x>0\tag{C0}\label{C0}.
\]

\begin{rem}
\upshape
The exponent $p$ in condition \eqref{C0} is necessarily greater than one by the convexity of
$\varphi$, which can easily be seen as follows. We have for arbitrary $\lambda,x,y\in(0,1]$
\[
\varphi(\lambda x+(1-\lambda)y) \le\lambda \varphi(x)+(1-\lambda)\varphi(y).
\]
Setting $x=1-s$ and $y=1$, we obtain
\[
\varphi(\lambda(1-s)+1-\lambda)=\varphi(1-\lambda s) \le\lambda \varphi(1-s)
\]
and, thus,
\[
\lim_{s\downarrow 0}\frac{\varphi(1-\lambda s)}{\varphi(1-s)} =\lambda^p\le \lambda.
\]
But this requires $p\ge 1$.
\end{rem}

A df $C_\varphi$, whose generator satisfies condition \eqref{C0}, is in the domain of attraction of
a multivariate extreme value distribution. Precisely, we have the following result.

\begin{prop}\label{prop:domain of attraction for Archimedean copulas}
Suppose that the generator $\varphi$ satisfies condition \eqref{C0}. Then we have
$C_\varphi\in\mathcal D(G)$, where $G(\bfx)=\exp\left(-\norm{\bfx}_p\right)$,
$\bfx\le\bfzero\in\R^d$.
\end{prop}

\begin{proof}
First we show that condition \eqref{C0} implies for $x>0$
\begin{equation}\label{eqn:property of inverse function}
\lim_{s\downarrow 0}\frac{1-\varphi^{-1}(sx)}{1-\varphi^{-1}(s)} = x^{1/p}.
\end{equation}
Choose $\delta_{sx},\delta_s\in(0,1)$ such that
\[
\varphi(1-\delta_{sx})=sx,\quad \varphi(1-\delta_{s})=s,
\]
i.e.,
\[
\varphi^{-1}(sx)= 1-\delta_{sx},\quad \varphi^{-1}(s)= 1-\delta_{s}.
\]

Condition \eqref{C0} implies for $s\downarrow 0$
\[
x=\frac{\varphi(1-\delta_{sx})}{\varphi(1-\delta_s)}=
\frac{\varphi\left(1-\delta_s\frac{\delta_{sx}} {\delta_s}\right)}{\varphi(1-\delta_s)}\sim
\left(\frac{\delta_{sx}}{\delta_s}\right)^p,
\]
where $\sim$ means that the ratio of the left hand side and the right hand side converges to one as
$s$ converges to zero. But this is
\[
\lim_{s\downarrow 0}\frac{1-\varphi^{-1}(sx)}{1-\varphi^{-1}(s)} = x^{1/p}.
\]

Next we show that for $\bfx=(x_1,\dots,x_d)\le\bfzero\in\R^d$
\[
\lim_{n\to\infty}C_\vartheta^n\left(\bfone+\frac{\bfx}n\right) =
\lim_{n\to\infty}\left[\varphi^{-1}\left(\sum_{i=1}^d\varphi\left(1+\frac{x_i}n\right) \right)
\right]^n =\exp\left(-\norm{\bfx}_p\right).
\]
Taking logarithms on both sides, this is equivalent with
\[
\lim_{n\to\infty} n\left[1-\varphi^{-1}\left(\sum_{i=1}^d\varphi\left(1+\frac{x_i}n\right)
\right)\right] = \norm{\bfx}_p.
\]
Write
\[
\frac 1n = 1-\varphi^{-1}\left(\varphi\left(1-\frac 1n\right)\right).
\]
Then
\begin{align*}
n\left[1-\varphi^{-1}\left(\sum_{i=1}^d\varphi\left(1+\frac{x_i}n\right) \right)\right] &=
\frac{1-\varphi^{-1}\left(\sum_{i=1}^d\varphi\left(1+\frac{x_i}n\right)
\right)}{1-\varphi^{-1}\left(\varphi\left(1-\frac 1n\right)\right)}\\
&= \frac{1-\varphi^{-1}\left(\varphi\left(1-\frac
1n\right)\sum_{i=1}^d\frac{\varphi\left(1+\frac{x_i}n\right)}{\varphi\left(1-\frac 1n \right)}
\right)}{1-\varphi^{-1}\left(\varphi\left(1-\frac 1n\right)\right)}\\
&\to_{n\to\infty} \left(\sum_{i=1}^d(-x_i)^p \right)^{1/p}
\end{align*}
by condition \eqref{C0} and equation \eqref{eqn:property of inverse function}, which is the
assertion.
\end{proof}

Condition \eqref{C0} on $\varphi$ is, for example, implied by the condition
\[
\lim_{s\downarrow 0} \frac{\varphi(1-s)}{s^p}=A\tag{C1}\label{C1}
\]
for some constant $A>0$ and $p\ge 1$, which is obviously satisfied by the generator
$\varphi_p(s)=(1-s)^p$.

Condition \eqref{C1} is by l'Hospital's rule implied by
\[
-\lim_{s\downarrow 0} \frac{\varphi'(1-s)}{s^{p-1}}=pA.\tag{C2}\label{C2}
\]
As a consequence, \eqref{C2} implies the condition
\[
-\lim_{s\downarrow 0} \frac{s\varphi'(1-s)}{\varphi(1-s)}=p.\tag{C3}\label{C3}
\]
\citet[Theorem 4.1]{chas09} showed, among others, that a copula $C_\varphi$, whose generator
satisfies \eqref{C3}, is in the domain of attraction of $G(\bfx)=\exp(-\norm{\bfx}_p)$,
$\bfx\le\bfzero\in\R^d$; see also \citet[Corollary 3.1.15]{falk2019}. In this case we have tail
independence only if $p=1$.

The Clayton family with generator $\varphi_\vartheta(t):= \left(t^{-\vartheta}-1\right)/\vartheta$
and $\vartheta > 0$, satisfies condition \eqref{C2} with $p=1$ and $A=1$. As a consequence, we have
independent tails for each $\vartheta>0$.

The Frank family has the generator
\[
\varphi_\vartheta(t):= -\log\left(\frac{e^{-\vartheta t}-1}{e^{-\vartheta}-1}\right),\qquad
\vartheta>0.
\]
It satisfies condition \eqref{C0} with $p=1$, i.e., we have again independent tails for each
$\vartheta>0$.

Consider, on the other hand, the generator $\varphi_\vartheta(t):=(-\log(t))^\vartheta$,
$\vartheta\ge 1$, of the Gumbel-Hougaard family of Archimedean copulas. This generator satisfies
condition \eqref{C0} with $p=\vartheta$ and, thus, we have tail independence only for
$\vartheta=1$.

\section{Main Theorem}\label{sec:main theorem}

In this section we establish conditional tail independence of the margins of $C_\varphi$, if the
generator $\varphi$ satisfies condition \eqref{C0}. First we compute the conditional df of
$(U_1,\dots,U_{j-1},U_{j+1},\dots,U_d)$, given that $U_j=u$.

\begin{lemma}\label{lem:representation of conditional df}
We have for $j\in\set{1,\dots,d}$ and
$\bfu=(u_1,\dots,u_{j-1},u,u_{j+1},\dots,u_d)\in[\bfu_0,\bfone)$
\begin{align*}
H_{j,u}(u_1,\dots,u_{j-1},u_{j+1},\dots,u_d)&:=P(U_i\le u_i,\,1\le i\le d,\, i\not=j\mid U_j=u)\\
 &= \frac{\varphi'(u)}{\varphi'(C(\bfu))}\\
&=\frac{\varphi'(u)}{\varphi'\left(\varphi^{-1}\left(\varphi(u)+\sum_{1\le i\le d,
\,i\not=j}\varphi(u_i) \right) \right)},
\end{align*}
provided the derivative $\varphi'(v)$ exists in a neighborhood of $u$, that $\varphi'$ is continuous
at $u$ with $\varphi'(u)\not=0$, and that $C(\bfu)\not=0$ as well.
\end{lemma}

\begin{proof} For notational simplicity we establish the result for the choice $j=d$.
We have for for $\bfu=(u_1,\dots,u_d)\in[\bfu_0,\bfone)$
\begin{align*}
&P(U_i\le u_i,\,1\le i\le d-1\mid U_d=u_d)\\
&=\lim_{\varepsilon\downarrow 0}\frac{P(U_i\le u_i,\,1\le i\le d-1,\, U_d\in[u_d,u_d+\varepsilon])}
{P(U_d\in[u_d,u_d+\varepsilon])}\\
&=\lim_{\varepsilon\downarrow 0} \frac{P(U_i\le u_i,\,1\le i\le d-1,\, U_d\le u_d+\varepsilon) -
P(U_i\le u_i,\,1\le i\le d-1,\, U_d\le u_d)}{\varepsilon}\\
&=\lim_{\varepsilon\downarrow 0}\frac{\varphi^{-1}\left(\sum_{i=1}^{d-1}\varphi(u_i)+
\varphi(u_d+\varepsilon) \right)-\varphi^{-1}\left(\sum_{i=1}^d\varphi(u_i) \right)}{\varepsilon}\\
&= \left(\varphi^{-1}\right)'\left(\sum_{i=1}^d\varphi(u_i) \right)\,\varphi'(u_d)\\
&= \frac{\varphi'(u_d)}{\varphi'\left(\varphi^{-1}\left(\sum_{i=1}^d\varphi(u_i) \right) \right)}\\
&= \frac{\varphi'(u_d)}{\varphi'(C_\varphi(\bfu))},
\end{align*}
which is the assertion.
\end{proof}

Note that the univariate margins of the df $H_{j,u}$, $1\le j\le d$, coincide in their upper tails,
where they are equal to
\[
H_u(v):=\frac{\varphi'(u)}{\varphi'\Big(\varphi^{-1}\left(\varphi(u)+\varphi(v) \right)\Big)},\qquad
v_0\le v\le 1,
\]
with $v_0=\max_{1\le i\le d}u_{0i}$.

The upper endpoint of $H_u$ is one, and, therefore, if the df $H_u$ is in the domain of attraction
of a univariate extreme value df $G$, then the family of negative Weibull distributions
$G_\alpha(x):=\exp\left(-\abs{x}^\alpha\right)$, $x\le 0$, with $\alpha>0$, is the first choice.
Note that $\alpha=1$ yields the standard negative exponential distribution.

The univariate df $H_u$ is in the domain of attraction of $G_\alpha$ for some $\alpha>0$ if and only
if (iff)
\[
\lim_{s\downarrow 0}\frac{1-H_u(1-sx)}{1-H_u(1-s)}=x^\alpha,\qquad x>0,
\]
see, e.g., \citet[Theorem 2.1.2]{gal87}.

\begin{lemma}\label{lem:characterization of univariate domain of attraction}
Suppose that the second derivative of $\varphi$ exists in a neighborhood of $u>v_0$, and that it is
continuous in $u$ with $\varphi''(u)\not=0\not=\varphi'(u)$. The univariate df $H_u$ satisfies
$H_u\in\mathcal D(G_p)$ for some $p\ge 1$ iff $\varphi$ satisfies condition \eqref{C0}.
\end{lemma}

\begin{proof}
Applying Taylor's formula twice shows that
\begin{align*}
1-H_u(1-s) &= \frac{\varphi'\Big(\varphi^{-1}\big(\varphi(u)+\varphi(1-s) \big) \Big)-\varphi'(u)}
{\varphi'\Big(\varphi^{-1}\big(\varphi(u)+\varphi(1-s) \big) \Big)}\\
&\sim \frac{\varphi''(u)}{\varphi'(u)^2}\varphi(1-s)
\end{align*}
as $s\downarrow 0$, which is the assertion.
\end{proof}

The next result is our main theorem.

\begin{theorem}\label{theo:main theorem}
Suppose the generator $\varphi$ of $C_\varphi$ satisfies condition \eqref{C0}. Then, if $u> u_{0j}$,
and $\varphi$ satisfies the differentiability conditions in Lemma \ref{lem:characterization of
univariate domain of attraction}, we obtain for $\bfx=(x_1,\dots,x_d)\le\bfzero\in\R^{d-1}$
\[
H_{j,u}\left(\bfone+ca_n\bfx\right)^n\to_{n\to\infty}\exp\left(-\sum_{i=1}^{d-1}(-x_i)^p \right),
\]
with $c:= \big(\varphi'(u)^2/\varphi''(u)\big)^{1/\alpha}$ and $a_n:=1-\varphi^{-1}(1/n)$, $n\ge
n_0$.
\end{theorem}

Note that the convexity of $\varphi$ implies that $\varphi''(u)\ge 0$.

\begin{rem}\upshape
The preceding result shows tail independence of $H_{j,u}$, as the limiting df is the product of its
margins.

Lemma \ref{lem:characterization of univariate domain of attraction} implies, moreover, that also the
reverse implication in the previous result holds, i.e., if $H_{j,u}$ is in the domain of attraction
of a multivariate max-stable df $G$ with negative Weibull margins having parameter at least one,
then condition \eqref{C0} is satisfied by Lemma \ref{lem:characterization of univariate domain of
attraction}, and $G$ has by the preceding result identical independent margins.

Finally, by the preceding arguments, we have $H_{j,u}\in\mathcal D(G)$, where $G$ has negative
Weibull margins, iff just one univariate margin of $H_{j,u}$ is in the domain of attraction of a
univariate extreme value distribution, and in this case $G$ has identical and independent margins.
\end{rem}

\begin{proof}
For notational simplicity we establish this result for $j=d$. It is sufficient to establish for
$\bfx=(x_1,\dots,x_d)\le\bfzero\in\R^{d-1}$
\begin{equation}\label{eqn:equivalent formulation of main result}
n\big(1-H_{d,u}\left(\bfone+ca_n\bfx\right)\big)\to_{n\to\infty} \sum_{i=1}^{d-1}(-x_i)^p.
\end{equation}

We know from Lemma \ref{lem:representation of conditional df} that for
$(u_1,\dots,u_{d-1},u)\in[\bfu_0,\bfone]$,
\begin{equation}\label{eq:conditional}
H_{d,u}(u_1,\dots,u_{d-1})=\frac{\varphi'(u)}{\varphi'\left(\varphi^{-1}\left(\varphi(u)+\sum_{i=1}^{d-1}\varphi(u_i)
\right) \right)}.
\end{equation}

As a consequence we obtain
\begin{align*}
&n\big(1-H_{d,u}\left(\bfone+ca_n\bfx\right)\big)\\
&= n\left(1-
\frac{\varphi'(u)}{\varphi'\left(\varphi^{-1}\left(\varphi(u)+\sum_{i=1}^{d-1}\varphi\left(1+ca_n
x_i\right) \right) \right)} \right)\\
&= n \frac{\varphi'\left(\varphi^{-1}\left(\varphi(u)+\sum_{i=1}^{d-1}\varphi\left(1+ca_n x_i\right)
\right)
\right)-\varphi'(u)}{\varphi'\left(\varphi^{-1}\left(\varphi(u)+\sum_{i=1}^{d-1}\varphi\left(1+ca_n
x_i\right) \right) \right)},
\end{align*}
where the denominator converges to $\varphi'(u)$ as $n$ increases.

Taylor's formula yields that the nominator equals
\[
\varphi''(\vartheta_n)\left(\varphi^{-1}\left(\varphi(u)+\sum_{i=1}^{d-1}\varphi\left(1+ca_n
x_i\right) \right)-u \right),
\]
where $\varphi''(\vartheta_n)$ converges to $\varphi''(u)$ as $n$ increases. Applying Taylor's
formula again yields
\begin{align*}
\varphi^{-1}\left(\varphi(u)+\sum_{i=1}^{d-1}\varphi\left(1+ca_n x_i\right) \right)-u
&=\frac 1{\varphi'\left(\varphi^{-1}(\xi_n)\right)}\sum_{i=1}^{d-1}\varphi\left(1+ca_n x_i \right),
\end{align*}
where $\xi_n$ converges to $\varphi(u)$ as $n$ increases. But
\[
n \sum_{i=1}^d\varphi\left(1+ca_n x_i\right) = \sum_{i=1}^d\frac{\varphi\left(1+ca_n
x_i\right)}{\varphi(1-a_n)} \to_{n\to\infty}\sum_{i=1}^{d-1}(-cx_i)^p.
\]
by condition \eqref{C0}. This yields the assertion.
\end{proof}

\section{Archimax Copulas}

Let $\varphi:(0,1]\to[0,\infty)$ be the generator of an Archimedean copula $C_\varphi(\bfu)=\varphi^{-1}\left(\sum_{i=1}^d\varphi(u_i)\right)$, $\bfu=(u_1,\dots,u_d)\in(0,1]^d$, and let $\norm\cdot_D$ be an arbitrary $D$-norm. Put
\begin{equation}\label{defn:definition of Archimax copula}
C(\bfu):=\varphi^{-1}\left(\norm{(\varphi(u_1),\dots,\varphi(u_d))}_D\right),\qquad \bfu\in(0,1]^d.
\end{equation}
It was established by \citet{char14} that $C$ actually defines a copula on $\R^d$, called \emph{Archimax copula}. Choosing $\norm\cdot_D=\norm\cdot_1$ yields $C(\bfu)=C_\varphi(\bfu)$ and, thus, the concept of Archimax copulas generalizes that of Archimedean copulas considerably.

To include also the generator family $\varphi_p(s)=(1-s)^p$, $s\in[0,1]$, $p\ge 1$, we require the representation of $C$ in equation \eqref{defn:definition of Archimax copula} only for $\bfu\in[\bfu_0,\bfone]\subset(0,1]^d$. There actually exists a rv, whose copula satisfies
\begin{equation*}
C(\bfu)=\varphi^{-1}\left(\norm{(\varphi(u_1),\dots,\varphi(u_d))}_p\right),\qquad \bfu\in[\bfu_0,\bfone]
\end{equation*}
with some $\bfu_0\in(0,1)^d$. This follows from the fact that $\norm{(\abs{x_1}^p,\dots,\abs{x_d}^p)}_D^{1/p}$ is again a $D$-norm, with an arbitrary $D$-norm $\norm\cdot_D$ and $p\ge 1$, see Proposition 2.6.1 and equations (2.14), (2.15) in \citet{falk2019}.

An Archimax copula is in the domain of attraction of a multivariate extreme value distribution, if the generator satisfies condition \eqref{C0}. Precisely, we have the following result.

\begin{prop}
Suppose the generator $\varphi$ satisfies condition \eqref{C0}. Then the corresponding Archimax copula $C$, with arbitrary $D$-norm $\norm\cdot_D$, satisfies $C\in\mathcal D(G)$, where $G(\bfx)=\exp\left(-\norm{\left(\abs{x_1}^p,\dots,\abs{x_d}^p\right)}_D^{1/p}\right)$, $\bfx\le\bfzero\in\R^d$.
\end{prop}

\begin{proof}
We have for $\bfx=(x_1,\dots,x_d)\le\bfzero\in\R^d$
\begin{align*}
&n\left[1-\varphi^{-1}\left(\norm{\left(\varphi\left(1+\frac{x_1}n\right),\dots,\varphi\left(1+\frac{x_d}n \right) \right)}_D \right) \right]\\
&= \frac{1-\varphi^{-1}\left(\varphi\left(1-\frac 1n \right)\norm{\left(\frac{\varphi\left(1+\frac{x_1}n \right)}{\varphi\left(1-\frac 1n \right)},\dots, \frac{\varphi\left(1+\frac{x_d}n \right)}{\varphi\left(\varphi\left(1-\frac 1n\right) \right)} \right)}_D \right)} {1-\varphi^{-1}\left(\varphi\left(1-\frac 1n\right) \right)}\\
&\to_{n\to\infty}\norm{\left(\abs{x_1}^p,\dots,\abs{x_d}^p \right)}_D^{1/p}
\end{align*}
by condition \eqref{C0} and equation \eqref{eqn:property of inverse function}. Repeating the arguments in the proof of Proposition \ref{prop:domain of attraction for Archimedean copulas} yields the assertion.
\end{proof}

Let the rv $\bfU=(U_1,\dots,U_d)$ follow an Archimax copula with generator function $\varphi$ and $D$-norm $\norm\cdot_D$. Does it also have independent tails, conditional on one of its components? We give a partial answer to this question.

Suppose the underlying  $\norm\cdot_D$ is a logistic one $\norm\cdot_q$, with $q\ge 1$. Then
\begin{align*}
\varphi^{-1}\left(\norm{(\varphi(u_1),\dots,\varphi(u_d))}_q\right) &= \varphi^{-1}\left(\left(\sum_{i=1}^d\varphi(u_i)^q\right)^{1/q}\right)\\
&= \psi^{-1}\left(\sum_{i=1}^d\psi(u_i) \right),
\end{align*}
where
\[
\psi(s):=\varphi(s)^q,\qquad s\in[0,1].
\]
If the generator $\varphi$ satisfies condition \eqref{C0}, then the generator $\psi$ clearly satisfies condition \eqref{C0} as well:
\[
\lim_{s\downarrow 0}\frac{\psi(1-sx)}{\psi(1-s)} = x^{pq},\qquad x >0.
\]
If $\varphi$ satisfies the differentiability conditions in Lemma \ref{lem:characterization of univariate domain of attraction}, then the conclusion of Theorem \ref{theo:main theorem} applies, i.e., with the choice $\norm\cdot_D=\norm\cdot_q$, $q\ge 1$, the rv $\bfU$ has again independent tails, conditional on one of its components.

Set, on the other hand $\bfU=(U,\dots,U)$, where $U$ is a rv that follows the uniform distribution on $(0,1)$. Choose $\norm\cdot_D=\norm\cdot_\infty$ with $\norm{\bfx}_\infty=\max_{1\le i\le d}(\abs{x_i})$. Then we have for every function $\varphi:(0,1]\to[0,\infty)$, which is continuous and strictly decreasing,
\begin{align*}
C(\bfu)&=P(U\le u_1,\dots,U\le u_d)\\
&=\min_{1\le i\le d}u_i\\
&=\varphi^{-1}\left(\norm{(\varphi(u_1),\dots,\varphi(u_d))}_\infty\right),\qquad \bfu\in (0,1]^d.
\end{align*}
The copula $C$ is, therefore, an Archimax copula, but it has completely dependent conditional margins.

\section{Simulation Study}
We conducted a simulation study to illustrate our findings on the conditional tail independence of the
Archimedean Gumbel-Hougaard copula family with dimension $d>2$ and dependence parameter $\vartheta>1$. The condition on $\vartheta$ implies that copula's tails are asymptotically dependent.  
There are several statistical tests to verify whether the tails of a multivariate distribution are asymptotically independent, provided that the latter is in the domain of attraction of a multivariate extreme value df.  
In the bivariate case, some tests have been suggested by \citet{draisma2004}, \citet{husler2009}, Chapter 6.5 in \citet{fahure10}. However, to extend them in higher dimensions than two is not straightforward. 
Therefore, we rely on the hypothesis testing proposed by \citet{guillou2018}, which is based on the componentwise maximum approach and is meant for an arbitrary dimension $d \geq 2$. 
Such a test is based on a system of hypotheses where under the null hypothesis it is assumed that $A(\bft)=1$ for all $\bft\in\calS_d$, i.e. the tails are asymptotically independent, while under the alternative hypothesis it is assumed that $A(\bft)<1$ for at least one $\bft\in\calS_d$, i.e. some tails are asymptotically dependent. Here, $A$ is the Pickands dependence function and $\calS_d$ is $d$-dimensional unit simplex \citep[e.g.,][Ch. 4]{fahure10}. In \citet{guillou2018} the authors proposed to use the test statistic $\widehat{S}_n=\sup_{\bft\in\calS_d}\sqrt{n}|\widehat{A}_n(\bft)-1|$ to decide whether or not to reject null hypothesis, where $\widehat{A}_n$ is an appropriated estimator of the the Pickands dependence function and $n$ is the sample size of the componentwise maxima. Under the null hypothesis, the test statistic converges to a suitable random variable $S$, for large samples. Large values of the observed test statistic  
provide evidence against the null hypothesis and in particular the quantiles of the distribution of $S$ to use for rejection of the null hypothesis are reported in Table 1 of \citet{guillou2018}.

We performed the following simulation experiment. In the first step we
simulated a sample of size $n=110$K of independent observations from a Gumbel-Hougaard copula 
with $d=3$ and $\vartheta=3$. Then, we computed the vector of normalized
componentwise maxima $m_{n,j}=\max_{i=1,\ldots,n}(u_{i,j}-b_{n,j})/a_{n,j}$ with $a_{n,j}=n$, $b_{n,j}=1$ and $j=1,\ldots,d$. In the second step, for
$u=0.99$ and $\varepsilon=0.0005$ we selected the observations $(u_{i,1},\ldots,u_{i,j-1},u_{i,j+1}, \ldots, u_{i,d})$ such that $u_{i,j}\in[u-\varepsilon,u+\varepsilon]$, $i=1,\ldots,n$. To work with a sample with fixed size we considered only $k=1000$ of such observations.
Then, we computed the vector of normalized
componentwise maxima $m^{*}_{k,s}=\max_{i=1,\ldots,k} u_{i,s}/(c a_{k,s})$, where $c= \big(\varphi'(u)^2/\varphi''(u)\big)^{1/\vartheta}$ and $a_{k,s}:=1-\varphi^{-1}(1/k)$ with $\varphi(t):=(-\log(t))^\vartheta$ and $s=1,\ldots,j-1,j-1,\ldots,d$. We repeated the first and second 
steps $N=100$ times obtaining two samples of componentwise maxima, one from the $d$-dimensional copula and one from the corresponding $d-1$ conditional distribution. 
The top-left and top-right panel of Figure \ref{fig:simulated_data} display an example of maxima obtained
from the Gumbel-Hougaard and the associated estimate of the Pickands dependence function, respectively. 
A strong dependence among the variables is evident. 
To see this better in the 
middle panels the maxima of a pair of variables and the relative estimate of the Pickands dependence function are reported. Indeed, the latter is close to lower bound $\max(1-t, t)$, i.e. the case of complete dependence. 
The bottom panels of Figure \ref{fig:simulated_data} display the maxima obtained with the second step of the simulation experiment and the associated estimate of the Pickands dependence function. These maxima, in contrast to the previous ones, seem to be independent and indeed the estimated Pickands dependence function is close to the upper bound (i.e. the case of independence).
Then, we applied the hypothesis test with the sample of maxima obtained in the first and second step of the simulation experiment, leading to the observed values of test statistic of
$3.843$ and $0.348$, respectively. Since the $0.95$-quantiles of the distribution of $S$ are $1.300$ and $0.960$ for $d=3$ and $d=2$, respectively \citep{guillou2018}, we conclude that we 
reject the hypothesis of tails independence
with the first sample of maxima whereas we do not reject it with the second sample. These results are consistent with our theoretical finding.

We repeated this simulation experiment $M=1000$ times and with the maxima obtained with the second step of the simulation experiment we computed the rejection rate of the null hypothesis. Since we simulated data under the null hypothesis we expect that the rejection rate is close the nominal value of the first type error, i.e. $5\%$. We did this for different dimension $d$ and values of the parameter $\vartheta$. The results are collected in Table \ref{tab:simulation}.
Again the simulation results show that our theoretical findings are correct.

%
%
%
\begin{figure}[t!]
	\centering
	\includegraphics[width=0.42\textwidth, page=1]{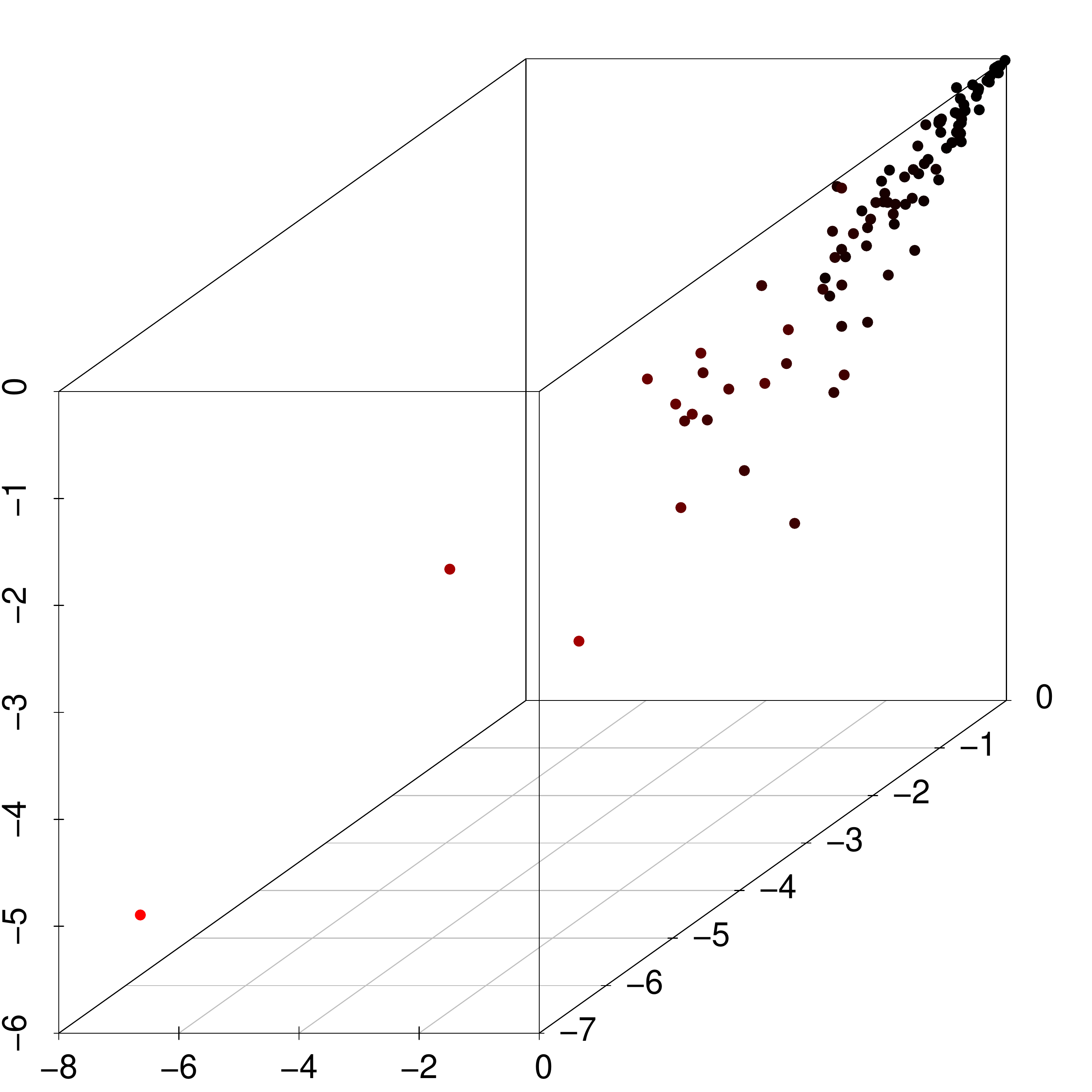}
	\includegraphics[width=0.42\textwidth, page=6]{Images/plots.pdf}\\
	\includegraphics[width=0.42\textwidth, page=2]{Images/plots.pdf}
	\includegraphics[width=0.42\textwidth, page=3]{Images/plots.pdf}
	\includegraphics[width=0.42\textwidth, page=4]{Images/plots.pdf}
	\includegraphics[width=0.42\textwidth, page=5]{Images/plots.pdf}
	\caption{Top-left panel displays the maxima obtain with the data simulated from a trivariate Gumbel-Hougaard copula with $\vartheta=4$. The middle one shows the maxima corresponding to two components. Finally, the one below shows the maxima obtain with the simulated data where one component is set to be a high value. The right-column report the relative estimated Pickands dependence function.}
	\label{fig:simulated_data}
\end{figure}
%
%

%
\begin{table}[t!]
\begin{center}
\caption{Rejection rate (in percentage) of the null hypothesis (asymptotic independent tails) based on $M=1000$ simulations.}
\begin{tabular}{ccccccc}
\toprule
Dimension & \multicolumn{6}{c}{Dependence parameter}\\
$d$ & $\vartheta:$ & $2$ & $3$ & $4$ & $5$ & $6$\\
\hline
$3$ &&$5.414$&$4.877$&$5.438$&$5.352$&$5.725$\\
$4$ &&$5.216$&$5.783$&$5.491$&$4.841$&$4.591$\\
$5$ &&$5.353$&$4.396$&$5.791$&$4.685$&$4.454$\\
\bottomrule
\end{tabular}
\label{tab:simulation}
\end{center}
\end{table}
\bibliographystyle{chicago}
\bibliography{evt}

\end{document}